\documentclass[12pt, twoside, leqno]{article}
\usepackage{amsmath,amsthm}
\usepackage{amssymb,latexsym}
\usepackage{enumerate}

\newtheorem{thm}{Theorem}[section]
\newtheorem{prop}[thm]{Proposition}

\newtheorem{lem}[thm]{Lemma}
\newtheorem{prob}[thm]{Problem}

\theoremstyle{definition}

\numberwithin{equation}{section}

\usepackage{mathrsfs}
\usepackage{CJK}
\usepackage{amsmath}
\usepackage{fancyhdr}
\usepackage{enumerate}

\begin{document}

\baselineskip=17pt

\title{The classification of some polynomial maps in dimension three}
\author{Yuan He\\ MOE-LCSM,\\ School of Mathematics and Statistics,\\ Hunan Normal University, Changsha 410081, China \\ \emph{E-mail:} heyuan0424@163.com \\ Dan Yan\footnote{ The author is supported by the Scientific Research Fund of Hunan Provincial Education Department (Grant No. 25B0088), the NSF of China (Grant No. 12371020) and the Construct Program of the Key Discipline in Hunan Province.}\\
MOE-LCSM,\\ School of Mathematics and Statistics,\\
 Hunan Normal University, Changsha 410081, China \\
\emph{E-mail:} yan-dan-hi@163.com \\
}
\date{}

\maketitle

\renewcommand{\thefootnote}{}

\renewcommand{\thefootnote}{\arabic{footnote}}
\setcounter{footnote}{0}

%%%%%%%%
\begin{abstract} In the paper, we classify all
polynomial maps of the form $H=(u(x,y),\allowbreak v(x,y,z),h(x,y,z))$ in the
case that $JH$ is nilpotent and $\deg_zv\geq 2\deg_zh$. Then we give the structure of $H=(u(x,y),v(x,y,z),h(x,y,z))$ if $JH$ is nilpotent and $\deg_zh\leq 3$.
\end{abstract}
 \vskip 2.5mm

\section{Introduction}

Throughout this paper, we will write $K$ for any field with
characteristic zero and $K[X]=K[x_1,x_2,\ldots,x_n]$ for the
polynomial algebra over $K$ with $n$ indeterminates. Let
$F=(F_1,F_2,\ldots,F_n):K^n\rightarrow K^n$ be a polynomial map,
that is, $F_i\in K[X]$ for all $1\leq i\leq n$. Let
$JF=(\frac{\partial F_i}{\partial x_j})_{n\times n}$ be the Jacobian
matrix of $F$. For a polynomial $h\in K[X]$, we abbreviate $\frac{\partial h}{\partial x_j}$ as $h_{x_j}$, and define $\deg_{x_i} f$ as the
highest degree of variable $x_i$ in $f$.
The Jacobian Conjecture (JC) raised by O.H. Keller in 1939 in
\cite{1} states that a polynomial map
$F: K^n\rightarrow K^n$ is invertible if the Jacobian
determinant $\det JF$ is a nonzero constant. This conjecture has
been attacked by many people from various research fields, but it is
still open, even for $n\geq 2$. Only the case $n=1$ is obvious. For
more information about the wonderful 70-year history, see \cite{2},
\cite{3}, and the references therein.

In 1980, S.S.S.Wang (\cite{4}) showed that the JC holds for all
polynomial maps of degree 2 in all dimensions (up to an affine
transformation). A powerful result is the reduction to degree
3, due to H.Bass, E.Connell and D.Wright (\cite{2}) in 1982 and
A.Yagzhev (\cite{5}) in 1980, which asserts that the JC is true if
it holds for all polynomial maps $X+H$, where $H$ is homogeneous
of degree 3. Thus, many authors studied these maps and this led to pose the
following problem.

 {\em (Homogeneous) Dependence Problem (DP for short).} Let $H=(H_1,\ldots,H_n): K^n\rightarrow K^n$ be a (homogeneous) polynomial map of degree $d$ such
that $JH$ is nilpotent and $H(0)=0$. Whether $H_1,\ldots,H_n$ are
linearly dependent over $K$?

The answer to the above problem is affirmative if rank$JH\leq 1$
(\cite{2}). In particular, this implies that the DP
has an affirmative answer in the case $n=2$. De Bondt and van den
Essen gave an affirmative answer to the above problem in the case that $H$
is homogeneous and $n=3$ (\cite{8}). For cubic homogeneous $H$, the case $n = 4$ has been solved affirmatively by Hubbers in \cite{7}, using techniques of \cite{6}. For cubic homogeneous $H$ with rank$JH = 2$, the DP has an affirmative answer for every
$n$. For quadratic $H$, the DP has an affirmative answer if
rank$\allowbreak JH \leq 2$ (see \cite{B2} or \cite[Theorem 3.4]{12}), in
particular if $n \leq 3$. For quadratic homogeneous $H$,
the DP has an affirmative answer in the case $n \leq 5$,
and several authors contributed to that result. See \cite[Appendix A]{HKM}
and \cite{XS5} for the case $n = 5$.

The first counterexamples to the DP were found by
van den Essen (\cite{9}, \cite[Theorem 7.1.7 (ii)]{3}). He
constructed counterexamples for all $n \geq 3$. In another paper
(\cite{E}), he constructed a quadratic counterexample for $n = 4$, which can be
generalized to arbitrary even degree (see \cite[Example 8.4.4]{3} for degree $4$).

 M. de Bondt was the
first who found homogeneous counterexamples (\cite{10}). He
constructed homogeneous counterexamples of degree $6$
for $n = 5$, homogeneous counterexamples of degree $4$ and $5$
for all $n \geq 6$, and cubic homogeneous counterexamples for all $n \geq 10$.
Homogeneous counterexamples of larger degrees can be made as well,
except for $n = 5$ and odd degrees. A cubic homogeneous counterexample
for $n = 9$ can be found in \cite{SFGZ}, see also \cite[Section 4.2]{HKM}.

In \cite{18}, Chamberland and van den Essen classified
all polynomial maps of the form
$$H=\big(u(x_1,x_2),v(x_1,x_2,x_3),h(u(x_1,x_2),v(x_1,x_2,x_3))\big)$$
with $JH$ nilpotent. The second author and Tang \cite{13} classified
all polynomial maps of the form
$H=\big(u(x_1,x_2),v(x_1,x_2,x_3),h(x_1,x_2,x_3)\big)$ with $JH$ nilpotent and $\deg_{x_3}h\leq 2$. In \cite{14}, the second author and de Bondt classified all polynomial maps
of the form
$$H=\big(H_1(x_1,x_2,\ldots,x_n),H_2(x_1,x_2),H_3(x_1,x_2,H_1),\ldots,H_n(x_1,x_2,H_1)\big)$$
with $JH$ nilpotent.
Casta\~{n}eda and van den Essen classified in \cite{CE} all polynomial maps
of the form
$$H=\big(u(x_1,x_2),u_2(x_1,x_2,x_3),u_3(x_1,x_2,x_4),\ldots,u_{n-1}(x_1,x_2,x_n),
h(x_1,x_2)\big)$$ with $JH$ nilpotent. The second author \cite{11} classified
all polynomial maps of the form
$H=\big(u(x_1,x_2,x_3),v(x_1,x_2,x_3),h(x_1,x_2)\big)$ with $JH$ nilpotent and $\deg_{x_3}v\leq 1$.

In the paper, we first classify all polynomial maps of the form
$H=(u(x,y),\allowbreak v(x,y,z), h(x,y,z))$ in the case that $JH$ is nilpotent and $\deg_zv\geq 3\deg_zh-1$ in section 2. Then, in section
3, we classify all polynomial maps of the form
$H=(u(x,y),\allowbreak v(x,y,z), h(x,y,z))$ in the case that $JH$ is nilpotent and $2\deg_zh\leq\deg_zv\leq 3\deg_zh-2$. In addition, we also give the structure of $H=(u(x,y),\allowbreak v(x,y,z), h(x,y,z))$ in the case that $JH$ is nilpotent and $\deg_zh\leq 3$. We also give some problems for further research in section 4.

Notation: For convenience, we use $x,y,z$ instead of $x_1,x_2,x_3$ in the paper. We view that the polynomials are in $K[x,y,z]$ with coefficients in $K[x,y]$ when we comparing the coefficients of the degree of $z$.

\section{The case of $\deg_zv\geq 3\deg_zh-1$}

In the section, we classify polynomial maps of the form
$H=(u(x,y),v(x,y,z),\allowbreak h(x,y,z))$ in the case that $JH$ is
nilpotent and $\deg_zv\geq 3\deg_zh-1$. We first prove several lemmas that we need.

\begin{lem} \label{lem2.1}
Let $H=(u(x,y),v(x,y,z),h(x,y,z))$ be a polynomial map of $K[x,y,z]$ with $u_y\neq 0$.
Suppose that $v=\sum_{i=0}^dv_iz^i$, $h=\sum_{j=0}^th_jz^j$ with $v_dh_t\neq 0$ and $v_i, h_j\in K[x,y]$ for all $1\leq i\leq d$, $1\leq j\leq t$. If $JH$ is nilpotent and $d\geq 2t$, then

$(1)$ $u_y,~v_d,~h_t\in K^*$, $v_{d-1}'\in K$ and $(v_{d+t-i})_x=\dfrac{(d+t-i+1)v_{d+t-i+1}v_{d-1}'}{dv_d}$ for $t+2\leq i\leq d-t+1$;

$(2)$ $(v_{2t-k})_x=-\dfrac{1}{u_{y}}\sum_{\gamma=1}^{k-1}(t-\gamma+1)(t-k+\gamma+1)h_{t-\gamma+1}h_{t-k+\gamma+1}+\\ \dfrac{(2t-k+1)v_{2t-k+1}v_{d-1}^{'}}{dv_{d}}$ for $2\leq k\leq t$; $(v_{2t-k})_x=-\dfrac{(2t-k+1)h_{2t-k+1}(2h_{1}+u_{x})}{u_{y}}-\dfrac{1}{u_{y}}\sum_{\gamma=k-t+1}^{t-1}(t-\gamma+1)(t-k+\gamma+1)h_{t-\gamma+1}h_{t-k+\gamma+1}+\dfrac{(2t-k+1)v_{2t-k+1}v_{d-1}^{'}}{dv_{d}}$
 for $t+1\leq k\leq 2t-1$; $v_{0x}=-\dfrac{u_{x}^{2}+h_{1}u_{x}+h_{1}^{2}}{u_{y}}+\dfrac{v_{1}v_{d-1}^{'}}{dv_{d}}$.
\end{lem}
\begin{proof}
Since $JH$ is nilpotent, we have
\begin{eqnarray}
-u_{x}-v_{y}=h_{z},\label{eq2.1}\\
-u_{x}v_{y}-v_{x}u_{y}-u_{x}^{2}-v_{y}^{2}=h_{y}v_{z},\label{eq2.2}\\
-u_{x}^{3}-2u_{x}u_{y}v_{x}-v_{x}v_{y}u_{y}=u_{y}v_{z}h_{x}.\label{eq2.3}
\end{eqnarray}
It follows from equation \eqref{eq2.1} that
\begin{equation}\label{eq2.4}
-u_{x}-(v_{dy}z^{d}+\cdots+v_{1y}z+v_{0y})=th_{t}z^{t-1}+\cdots+2h_{2}z+h_{1}.
\end{equation}
Then we have $v_t,v_{t+1},\ldots,v_d\in K[x]$ and
\begin{eqnarray}
(v_{t-1})_y=-th_{t},\notag\\
	(v_{t-2})_y=-(t-1)h_{t-1},\notag\\
	\vdots~~~~~~~~~~\label{eq2.5}\\
	v_{1y}=-2h_{2},\notag\\
	-u_{x}-v_{0y}=h_{1}\notag
\end{eqnarray}
by comparing the coefficients of the degree of $z$ of equation \eqref{eq2.4}. It follows from equations \eqref{eq2.2}, \eqref{eq2.5} that
\begin{equation}\label{eq2.6}
\begin{split}
-u_{x}((v_{t-1})_yz^{t-1}+\cdots+v_{0y})-u_{y}(v_{dx}z^{d}+\cdots+v_{0x})-u_{x}^{2}-\\( (v_{t-1})_yz^{t-1}+\cdots+v_{0y}) ^{2}
=(h_{ty}z^{t}+\cdots+h_{0y})( dv_{d}z^{d-1}+\cdots+v_{1}).
\end{split}
\end{equation}
Since $d\geq 2t>2t-2$, we have
\begin{equation}\label{eq2.7}
h_t,h_{t-1},\ldots,h_2\in K[x]
\end{equation}
by comparing the coefficients of $z^{d+t-1}, z^{d+t-2},\ldots,z^{d+1}$ of equation \eqref{eq2.6}, respectively. Comparing the coefficients of $z^d$ of equation \eqref{eq2.6}, we have
\begin{equation}\label{eq2.8}
h_{1y}=-\frac{v_d'}{dv_d}u_y.
\end{equation}
Then we have
\begin{equation}\label{eq2.9}
h_1=-\frac{v_d'}{dv_d}u+c_1(x)
\end{equation}
with $c_1(x)\in K[x]$ by integrating both sides of equation \eqref{eq2.6} with respect to $y$. Consider the coefficients of $z^{d-1}$ of equation \eqref{eq2.6} and substitute equation \eqref{eq2.8} to it, we have
\begin{equation}\label{eq2.10}
h_{0y}=\left( \frac{(d-1)v_{d-1}v_{d}^{'}}{(dv_{d})^{2}}-\frac{v_{d-1}^{'}}{dv_{d}}\right) u_{y}.
\end{equation}
Then we have
\begin{equation}\label{eq2.11}
h_{0}=\left( \frac{(d-1)v_{d-1}v_{d}^{'}}{(dv_{d})^{2}}-\frac{v_{d-1}^{'}}{dv_{d}}\right) u+c_{2}(x)
\end{equation}
with $c_2(x)\in K[x]$ by integrating both sides of equation \eqref{eq2.10} with respect to $y$.\\

It follows from equations \eqref{eq2.3}, \eqref{eq2.5} that
\begin{equation}\label{eq2.12}
\begin{split}
-u_{x}^{3}-2u_{x}u_{y}(v_{dx}z^{d}+\cdots+v_{0x})-u_{y}(v_{dx}z^{d}+\cdots+v_{0x})( (v_{t-1})_yz^{t-1}+\\\cdots+v_{0y})
	=u_{y}(h_{tx}z^{t}+\cdots+h_{0x})(dv_{d}z^{d-1}+\cdots+v_{1}).
\end{split}
\end{equation}
Then we have
$$dv_dh_{tx}=th_tv_{dx}$$
by comparing the coefficients of $z^{d+t-1}$ of equation \eqref{eq2.12}. Thus, we have
\begin{equation}\label{eq2.13}
v_d^t=\lambda h_t^d
\end{equation}
with $\lambda\in K^*$ by integrating both sides of the above equation with respect to $x$.
Comparing the coefficients of $z^d$ of equation \eqref{eq2.12}, we have
\begin{equation}
\begin{split}
dv_{d}h_{1x}+(d-1)v_{d-1}h_{2x}+\cdots+(d-t+1)v_{d-t+1}h_{tx}\notag\\=-2u_{x}v_{dx}-v_{dx}v_{0y}-(v_{d-1})_xv_{1y}-\cdots-(v_{d-t+1})_x(v_{t-1})_y.
\end{split}
\end{equation}
Substituting equations \eqref{eq2.5}, \eqref{eq2.9} to the above equation, we have
\begin{equation}
\begin{split}
\frac{(d+1)(v_{d}^{'})^{2}-dv_{d}v_{d}^{''}}{dv_{d}}u+dv_{d}c_{1}^{'}(x)+(d-1)v_{d-1}h_{2}^{'}+\cdots\\+(d-t+1)v_{d-t+1}h_{t}^{'}\notag=v_{d}^{'}c_{1}(x)+2h_{2}v_{d-1}^{'}+\cdots+th_{t}v_{d-t+1}^{'}.
\end{split}
\end{equation}
Then we have $v_{d}^{'}=0$ by comparing the coefficients of $y$ of the above equation. Thus, we have $v_d\in K^*$. It follows from equation \eqref{eq2.13} that $h_t\in K^*$. Then the above equation has the following form:
\begin{equation}\label{eq2.14}
\begin{split}
dv_{d}c_{1}^{'}(x)+(d-1)v_{d-1}h_{2}^{'}+\cdots+(d-t+2)v_{d-t+2}h_{t-1}^{'}\\=2h_{2}v_{d-1}^{'}+3h_3v_{d-2}^{'}+\cdots+th_{t}v_{d-t+1}^{'}.
\end{split}
\end{equation}
Comparing the coefficients of $z^{d-1}$ of equation \eqref{eq2.12}, we have
\begin{equation}
\begin{split}
dv_{d}h_{0x}+(d-1)v_{d-1}h_{1x}+\cdots+(d-t+1)v_{d-t+1}(h_{t-1})_x\notag\\=-2u_{x}(v_{d-1})_x-(v_{d-1})_xv_{0y}-(v_{d-2})_xv_{1y}-\cdots-(v_{d-t})_x(v_{t-1})_y.
\end{split}
\end{equation}
Substituting equations \eqref{eq2.5}, \eqref{eq2.9}, \eqref{eq2.11} to the above equation, we have
\begin{equation}
\begin{split}
-v_{d-1}^{''}u+dv_{d}c_{2}^{'}(x)+(d-1)v_{d-1}c_{1}^{'}(x)+\cdots+(d-t+1)v_{d-t+1}h_{t-1}^{'}\notag\\=c_{1}(x)v_{d-1}^{'}+2h_{2}v_{d-2}^{'}+\cdots+th_{t}v_{d-t}^{'}.
\end{split}
\end{equation}
Then we have $v_{d-1}^{''}=0$ by comparing the coefficients of $y$ of the above equation. Thus, we have $v_{d-1}^{'}\in K$. Then the above equation has the following form:
\begin{equation}\label{eq2.15}
\begin{split}
dv_{d}c_{2}^{'}(x)+(d-1)v_{d-1}c_{1}^{'}(x)+\cdots+(d-t+1)v_{d-t+1}h_{t-1}^{'}\\=c_{1}(x)v_{d-1}^{'}+2h_{2}v_{d-2}^{'}+\cdots+th_{t}v_{d-t}^{'}.
\end{split}
\end{equation}\\

Since $v_d,~h_t\in K^*$, it follows from equations \eqref{eq2.9}, \eqref{eq2.11} that
\begin{eqnarray}
h_{1}=c_{1}(x) \in K[x],\label{eq2.16}\\
h_{0}= -\frac{v_{d-1}^{'}}{dv_{d}} u+c_{2}(x),\,c_{2}(x)\in K[x].\label{eq2.17}
\end{eqnarray}
Consider the coefficients of $z^{d+t-i}$$(t+2\leq i\leq d-t+1)$ of equation \eqref{eq2.6} and substitute equations \eqref{eq2.7}, \eqref{eq2.16} to it, we have
$$(d+t-i+1)v_{d+t-i+1}h_{0y}=-u_{y}(v_{d+t-i})_x.$$
Substituting equation \eqref{eq2.17} to the above equation, we have
\begin{equation}\label{eq2.18}
(v_{d+t-i})_x=\dfrac{(d+t-i+1)v_{d+t-i+1}v_{d-1}^{'}}{dv_{d}},\;(t+2\leq i\leq d-t+1).
\end{equation}
Comparing the coefficients of $z^{2t-2}$ of equation \eqref{eq2.6}, we have
$$(2t-1)v_{2t-1}h_{0y}=-u_y(v_{2t-2})_x-((v_{t-1})_y)^2.$$
Substituting equations \eqref{eq2.5}, \eqref{eq2.17} to the above equation, we have
$$\left( v_{2t-2}^{'}-\frac{(2t-1)v_{2t-1}v_{d-1}^{'}}{dv_{d}}\right)u_{y}=-t^{2}h_{t}^{2}.$$
Since $th_t\in K^*$, we have $u_y\in K^*$ and $v_{2t-2}^{'}=-\frac{t^{2}h_{t}^{2}}{u_{y}}+ \frac{\left( 2t-1\right) v_{2t-1}v_{d-1}^{'}}{dv_{d}}$. Consider the coefficients of $z^{2t-k}$$(3\leq k\leq t)$ of equation \eqref{eq2.6} and substitute equation \eqref{eq2.17} to it, we have
$$\left( \frac{\left( 2t-k+1\right) v_{2t-k+1}v_{d-1}^{'}}{dv_{d}}-v_{2t-k}^{'}\right) u_{y}=\sum_{\gamma=1}^{k-1}(v_{t-\gamma})_y(v_{t-k+\gamma})_y.$$
Since $u_y\in K^*$, it follows from the above equation and equation \eqref{eq2.5} that
\begin{equation}\label{eq2.19}
\begin{split}
v_{2t-k}^{'}=-\dfrac{1}{u_{y}}\sum_{\gamma=1}^{k-1}(t-\gamma+1)(t-k+\gamma+1)h_{t-\gamma+1}h_{t-k+\gamma+1}\\+\dfrac{(2t-k+1)v_{2t-k+1}v_{d-1}^{'}}{dv_{d}}
\end{split}
\end{equation}
for $2\leq k\leq t$. Consider the coefficients of $z^{2t-k}$$(t+1\leq k\leq 2t-1)$ of equation \eqref{eq2.6} and substitute equation \eqref{eq2.17} to it, we have
\begin{equation}\nonumber
\begin{split}
\left(\frac{(2t-k+1) v_{2t-k+1}v_{d-1}^{'}}{dv_{d}}-(v_{2t-k})_x\right) u_{y}=u_{x}(v_{2t-k})_y+2(v_{2t-k})_yv_{0y}\\+\sum_{\gamma=k-t+1}^{t-1}(v_{t-\gamma})_y(v_{t-k+\gamma})_y.
\end{split}
\end{equation}
Since $u_y\in K^*$, it follows from equation \eqref{eq2.5} that
\begin{equation}\label{eq2.20}
\begin{split}
(v_{2t-k})_x=-\dfrac{(2t-k+1)h_{2t-k+1}(2h_{1}+u_{x})}{u_{y}}+\dfrac{(2t-k+1)v_{2t-k+1}v_{d-1}^{'}}{dv_{d}}\\-\dfrac{1}{u_{y}}\sum_{\gamma=k-t+1}^{t-1}(t-\gamma+1)(t-k+\gamma+1)h_{t-\gamma+1}h_{t-k+\gamma+1}.
\end{split}
\end{equation}
Comparing the coefficients of $z^0$ of equation \eqref{eq2.6}, we have
$$v_{1}h_{0y}=-u_{y}v_{0x}-u_{x}v_{0y}-u_{x}^{2}-v_{0y}^{2}.$$
Substituting equations \eqref{eq2.5}, \eqref{eq2.17} to the above equation, we have
$$v_{0x}=-\dfrac{u_{x}^{2}+h_{1}u_{x}+h_{1}^{2}}{u_{y}}+\dfrac{v_{1}v_{d-1}^{'}}{dv_{d}}.$$
Then the conclusion follows.
\end{proof}

\begin{lem}\label{lem2.2}
Let $H=(u(x,y),v(x,y,z),h(x,y,z))$ be a polynomial map of $K[x,y,z]$ with $u_y\neq 0$.
Suppose that $v=\sum_{i=0}^dv_iz^i$, $h=\sum_{j=0}^th_jz^j$ with $v_dh_t\neq 0$ and $v_i, h_j\in K[x,y]$ for all $1\leq i\leq d$, $1\leq j\leq t$. If $JH$ is nilpotent and $d\geq 3t-1$, then
$c_{2}^{'}(x)=\dfrac{c_{1}(x)v_{d-1}^{'}}{dv_{d}},\,c_{1}^{'}(x)=\dfrac{2h_{2}v_{d-1}^{'}}{dv_{d}},\;h_{t-\alpha+1}^{'}=\dfrac{(t-\alpha+2)h_{t-\alpha+2}v_{d-1}^{'}}{dv_{d}}$
with $2\leq \alpha\leq t-1$.
\end{lem}
\begin{proof}
Since $JH$ is nilpotent, we have
\begin{eqnarray}
-u_{x}-v_{y}=h_{z},\label{eq2.21}\\
-u_{x}v_{y}-v_{x}u_{y}-u_{x}^{2}-v_{y}^{2}=h_{y}v_{z},\label{eq2.22}\\
-u_{x}^{3}-2u_{x}u_{y}v_{x}-v_{x}v_{y}u_{y}=u_{y}v_{z}h_{x}.\label{eq2.23}
\end{eqnarray}
Next we prove the conclusion by induction on $\alpha$. It follows from Lemma \ref{lem2.1} that $v_d,~h_t\in K^*$. If $\alpha=2$, then we have
$$dv_{d}h_{t-1}^{'}=-v_{d-1}^{'}(v_{t-1})_y$$
by comparing the coefficients of $z^{d+t-2}$ of equation \eqref{eq2.12}. Substituting equation \eqref{eq2.5} to the above equation, we have
\begin{equation}\label{eq2.24}
h_{t-1}^{'}=\frac{th_{t}v_{d-1}^{'}}{dv_{d}}
\end{equation}
Suppose that
\begin{equation}\label{eq2.25}
h_{t-\kappa+1}^{'}=\dfrac{(t-\kappa+2)h_{t-\kappa+2}v_{d-1}^{'}}{dv_{d}}
\end{equation}
for $2\leq \kappa\leq \alpha-1$. We have
\begin{equation}\nonumber
\begin{split}
dv_{d}(h_{t-\alpha+1})_x+(d-1)v_{d-1}(h_{t-\alpha+2})_x+\cdots+(d-\alpha+2)v_{d-\alpha+2}(h_{t-1})_x\\=-(v_{d-1})_x(v_{t-\alpha+1})_y-(v_{d-2})_x(v_{t-\alpha+2})_y-\cdots-(v_{d-\alpha+1})_x(v_{t-1})_y
\end{split}
\end{equation}
by comparing the coefficients of $z^{d+t-\alpha}$$(2\leq\alpha\leq t-1)$ of equation \eqref{eq2.12}.
Substituting equation \eqref{eq2.5} to the above equation, we have
\begin{equation}\label{eq2.26}
\begin{split}
dv_{d}h_{t-\alpha+1}^{'}+(d-1)v_{d-1}h_{t-\alpha+2}^{'}+\cdots+(d-\alpha+2)v_{d-\alpha+2}h_{t-1}^{'}\\=(t-\alpha+2)h_{t-\alpha+2}(v_{d-1})_x+(t-\alpha+3)h_{t-\alpha+3}(v_{d-2})_x+\cdots+th_{t}(v_{d-\alpha+1})_x.
\end{split}
\end{equation}
Since $d\geq 3t-1$ and $2\leq \alpha\leq t-1$, we have $d-\alpha+1\geq 2t+1>2t-2$. It follows from Lemma \ref{lem2.1} that
\begin{equation}\label{eq2.27}
(v_{d-\alpha+1})_x=\dfrac{(d-\alpha+2)v_{d-\alpha+2}v_{d-1}^{'}}{dv_{d}}
\end{equation}
for $2\leq \alpha\leq t-1$. Substituting equations \eqref{eq2.25}, \eqref{eq2.27} to equation \eqref{eq2.26}, we have
\begin{equation}\label{eq2.28}
h_{t-\alpha+1}^{'}=\dfrac{(t-\alpha+2)h_{t-\alpha+2}v_{d-1}^{'}}{dv_{d}}
\end{equation}
for $2\leq \alpha\leq t-1$. It follows from equation \eqref{eq2.14} that
\begin{equation}\nonumber
\begin{split}
dv_{d}c_{1}^{'}(x)+(d-1)v_{d-1}h_{2}^{'}+\cdots+(d-t+2)v_{d-t+2}h_{t-1}^{'}\\=2h_{2}v_{d-1}^{'}+3h_{3}v_{d-2}^{'}+\cdots+th_{t}v_{d-t+1}^{'}.
\end{split}
\end{equation}
Since $d\geq 3t-1$, we have $d-t+1\geq 2t$. It follows from Lemma \ref{lem2.1} that $v_{d-t+1}^{'}=\frac{(d-t+2)v_{d-t+2}v_{d-1}^{'}}{dv_{d}}$. Substituting equations \eqref{eq2.27}, \eqref{eq2.28} to the above equation, we have
\begin{equation}\label{eq2.29}
c_{1}^{'}(x)=\dfrac{2h_{2}v_{d-1}^{'}}{dv_{d}}.
\end{equation}
It follows from equation \eqref{eq2.15} that
\begin{equation}\nonumber
\begin{split}
dv_{d}c_{2}^{'}(x)+(d-1)v_{d-1}c_{1}^{'}(x)+\cdots+(d-t+1)v_{d-t+1}h_{t-1}^{'}\\=c_{1}(x)v_{d-1}^{'}+2h_{2}v_{d-2}^{'}+\cdots+th_{t}v_{d-t}^{'}.
\end{split}
\end{equation}
Since $d\geq 3t-1$, we have $d-t\geq 2t-1$. It follows from Lemma \ref{lem2.1} that $v_{d-t}^{'}=\frac{(d-t+1)v_{d-t+1}v_{d-1}^{'}}{dv_{d}}$. Substituting equations \eqref{eq2.27}, \eqref{eq2.28}, \eqref{eq2.29} to the above equation, we have
$$c_{2}^{'}(x)=\dfrac{c_{1}(x)v_{d-1}^{'}}{dv_{d}}.$$
Then the conclusion follows.
\end{proof}

\begin{thm}\label{thm2.3}
Let $H=(u(x,y),v(x,y,z),h(x,y,z))$ be a polynomial map of $K[x,y,z]$ with $H(0)= 0$. Suppose that $u,~v,~h$ are linearly independent over $K$. If $JH$ is nilpotent and $\deg_zv\geq 3\deg_zh-1$, then $u=g(ay+b(x))$, $v=v_1z-a^{-1}b'(x)g(ay+b(x))-v_1l_2x$ and $h=c_0u^2+l_2u$ with $b(x)=v_1c_0ax^2+l_1x+\tilde{l}_2$, $v_1,~c_0,~a\in K^*$, $l_1,~l_2,~\tilde{l}_2\in K$, $g(t)\in K[t]$ and $g(0)=0$, $\deg g(t)\geq 1$.
\end{thm}
\begin{proof}
Suppose that $v=\sum_{i=0}^dv_iz^i$, $h=\sum_{j=0}^th_jz^j$ with $v_dh_t\neq 0$ and $v_i, h_j\in K[x,y]$ for all $1\leq i\leq d$, $1\leq j\leq t$.

If $t= 0$, then the conclusion follows from Theorem 2.8 in \cite{13} or Theorem 2.5 in \cite{14}. Thus, we can assume that $t\geq 1$ in the following arguments.

If $u_y=0$, then it follows from Proposition 2.2 in \cite{13} that $u,~v,~h$ are linearly dependent over $K$, which is a contradiction. Thus, we have $u_y\neq 0$.

Since $JH$ is nilpotent, we have
\begin{eqnarray}
-u_{x}-v_{y}=h_{z},\label{eq2.30}\\
-u_{x}v_{y}-v_{x}u_{y}-u_{x}^{2}-v_{y}^{2}=h_{y}v_{z},\label{eq2.31}\\
-u_{x}^{3}-2u_{x}u_{y}v_{x}-v_{x}v_{y}u_{y}=u_{y}v_{z}h_{x}.\label{eq2.32}
\end{eqnarray}
It follows from equation \eqref{eq2.30} that
\begin{equation}\label{eq2.33}
-u_{x}-(v_{dy}z^{d}+\cdots+v_{1y}z+v_{0y})=th_{t}z^{t-1}+\cdots+2h_{2}z+h_{1}.
\end{equation}
Then we have $v_t,v_{t+1},\ldots,v_d\in K[x]$ and
\begin{eqnarray}
(v_{t-1})_y=-th_{t},\notag\\
	(v_{t-2})_y=-(t-1)h_{t-1},\notag\\
	\vdots~~~~~~~~~~\label{eq2.34}\\
	v_{1y}=-2h_{2},\notag\\
	-u_{x}-v_{0y}=h_{1}\notag
\end{eqnarray}
by comparing the coefficients of the degree of $z$ of equation \eqref{eq2.33}. It follows from equations \eqref{eq2.32}, \eqref{eq2.34} that
\begin{equation}\label{eq2.35}
\begin{split}
-u_{x}^{3}-2u_{x}u_{y}(v_{dx}z^{d}+\cdots+v_{0x})-u_{y}(v_{dx}z^{d}+\cdots+v_{0x})( (v_{t-1})_yz^{t-1}+\\\cdots+v_{0y})
	=u_{y}(h_{tx}z^{t}+\cdots+h_{0x})(dv_{d}z^{d-1}+\cdots+v_{1}).
\end{split}
\end{equation}
Then we have
$$u_{y}v_{1}h_{0x}=-u_{x}^{3}-2u_{x}u_{y}v_{0x}-u_{y}v_{0y}v_{0x}$$
by comparing the coefficients of $z^0$ of equation \eqref{eq2.35}. Substituting equations \eqref{eq2.17}, \eqref{eq2.34} to the above equation, we have
\begin{equation}\label{eq2.36}
u_{y}v_{1}\left( \frac{-v_{d-1}^{'}}{dv_{d}}u_{x}+c^{'}_{2}(x)\right) =-u_{x}^{3}+u_{y}v_{0x}(h_{1}-u_{x}).
\end{equation}
It follows from Lemma \ref{lem2.1} and Lemma \ref{lem2.2} that $v_{0x}=-\frac{u_{x}^{2}+h_{1}u_{x}+h_{1}^{2}}{u_{y}}+\frac{v_{1}v_{d-1}^{'}}{dv_{d}}$, $c_1(x)=h_1(x)$ and $c_{2}^{'}(x)=\frac{c_{1}(x)v_{d-1}^{'}}{dv_{d}}$. Substituting these equations to equation \eqref{eq2.36}, we have $h_1=0$. Next we prove that $h_i=0$ for all $2\leq i\leq t$ by induction on $i$.
Suppose that $h_1=h_2=\cdots=h_{i-1}=0$. We need to prove that $h_i=0$.

$(1)$ If $2<3i-3<t$, then we have
\begin{equation}\label{eq2.37}
\begin{split}
(3i-2)v_{3i-2}h_{0x}+(3i-3)v_{3i-3}h_{1x}+\cdots+v_{1}(h_{3i-3})_x\\=(h_{1}-u_{x})(v_{3i-3})_x+2h_{2}(v_{3i-4})_x+\cdots+(3i-2)h_{3i-2}v_{0x}
\end{split}
\end{equation}
by comparing the coefficients of $z^{3i-3}$ of equation \eqref{eq2.35} and substituting equation \eqref{eq2.34} to it. Substituting the equations in Lemma \ref{lem2.1}, Lemma \ref{lem2.2} and equation  \eqref{eq2.11} to equation \eqref{eq2.37}, we have
\begin{equation}\nonumber
\begin{split}
-\frac{3h_{1}^{2}\cdot(3i-2)h_{3i-2}}{u_{y}}-\frac{3h_{1}}{u_{y}}\sum_{\gamma=t-3i+4}^{t-1}(t-\gamma+1)(\gamma-t+3i-2)h_{t-\gamma+1}h_{\gamma-t+3i-2}\\-\frac{2h_{2}}{u_{y}}\sum_{\gamma=t-3i+5}^{t-1}(t-\gamma+1)(\gamma-t+3i-3)h_{t-\gamma+1}h_{\gamma-t+3i-3}-\cdots\\
	 -\frac{ih_{i}}{u_{y}}\sum_{\gamma=t-2i+3}^{t-1}(t-\gamma+1)(\gamma-t+2i-1)h_{t-\gamma+1}h_{\gamma-t+2i-1} \\-\cdots-\frac{(3i-4)h_{3i-4}}{u_{y}}(2h_{2})^{2}=0.
\end{split}
\end{equation}
Since $h_1=h_2=\cdots=h_{i-1}=0$ by induction hypothesis, we have $-\frac{(ih_{i})^{3}}{u_{y}}=0$ by substituting them to the above equation. Thus, we have $h_i=0$.

$(2)$ If $3i-3>2t-2$, then we have
\begin{equation}\label{eq2.38}
\begin{split}
(3i-2)v_{3i-2}h_{0x}+(3i-3)v_{3i-3}h_{1x}+\cdots+(3i-t-1)v_{3i-t-1}(h_{t-1})_x\\=(h_{1}-u_{x})(v_{3i-3})_x+2h_{2}(v_{3i-4})_x+\cdots+th_{t}(v_{3i-t-2})_x
\end{split}
\end{equation}
by comparing the coefficients of $z^{3i-3}$ of equation \eqref{eq2.35} and substituting equation \eqref{eq2.34} to it. Substituting the equations in Lemma \ref{lem2.1}, Lemma \ref{lem2.2} and equation  \eqref{eq2.11} to equation \eqref{eq2.38}, we have
\begin{equation}\nonumber
\begin{split}
-\frac{(3i-2t)h_{3i-2t}(th_{t})^{2}}{u_{y}}-\frac{(3i-2t+1)h_{3i-2t+1}}{u_{y}}[th_{t}(t-1)h_{t-1}+(t-1)h_{t-1}th_{t}]\\-\cdots-\frac{ih_{i}}{u_{y}}\sum_{\gamma=1}^{2t-2i+1}(t-\gamma+1)(\gamma-t+2i-1)h_{t-\gamma+1}h_{\gamma-t+2i-1}\\
-\frac{(i+1)h_{i+1}}{u_{y}}\sum_{\gamma=1}^{2t-2i+2}(t-\gamma+1)(\gamma-t+2i-2)h_{t-\gamma+1}h_{\gamma-t+2i-2}-\cdots\\-\frac{th_{t}}{u_{y}}\sum_{\gamma=1}^{3t-3i+1}(t-\gamma+1)(\gamma-2t+3i-1)h_{t-\gamma+1}h_{\gamma-2t+3i-1}=0.
\end{split}
\end{equation}
Since $h_1=h_2=\cdots=h_{i-1}=0$ by induction hypothesis, we have $-\frac{(ih_{i})^{3}}{u_{y}}=0$ by substituting them to the above equation. Thus, we have $h_i=0$.

$(3)$ If $2t+6\leq 6i< 3t+6$, then we have
\begin{equation}\label{eq2.39}
\begin{split}
-\frac{3h_{1}}{u_{y}}\sum_{\gamma=1}^{2t-3i+2}(t-\gamma+1)(\gamma-t+3i-2)h_{t-\gamma+1}h_{\gamma-t+3i-2}\\-\frac{2h_{2}}{u_{y}}\sum_{\gamma=1}^{\gamma-t-3i+3}(t-\gamma+1)(\gamma-t+3i-3)h_{t-\gamma+1}h_{\gamma-t+3i-3}-\cdots\\
	 -\frac{(3i-t-2)h_{3i-t-2}}{u_{y}}\sum_{\gamma=1}^{t-1}(t-\gamma+1)(\gamma+1)h_{t-\gamma+1}h_{\gamma+1}\\-\frac{(3i-t-1)h_{3i-t-1}}{u_{y}}\sum_{\gamma=2}^{t-1}(t-\gamma+1)\gamma h_{t-\gamma+1}h_{\gamma}-\cdots\\
	 -\frac{ih_{i}}{u_{y}}\sum_{\gamma=t-2i+3}^{t-1}(t-\gamma+1)(\gamma-t+2i-1)h_{t-\gamma+1}h_{\gamma-t+2i-1} -\cdots\\-\frac{th_{t}}{u_{y}}\sum_{\gamma=2t-3i+3}^{t-1}(t-\gamma+1)(\gamma-2t+3i-1)h_{t-\gamma+1}h_{\gamma-2t+3i-1}=0
\end{split}
\end{equation}
by substituting the equations in Lemma \ref{lem2.1}, Lemma \ref{lem2.2} and equation  \eqref{eq2.11} to equation \eqref{eq2.38}.

$(4)$ If $3t+6\leq 6i\leq 4t+2$, then we have
\begin{equation}\label{eq2.40}
\begin{split}
-\frac{3h_{1}}{u_{y}}\sum_{\gamma=1}^{2t-3i+2}(t-\gamma+1)(\gamma-t+3i-2)h_{t-\gamma+1}h_{\gamma-t+3i-2}\notag\\-\frac{2h_{2}}{u_{y}}\sum_{\gamma=1}^{\gamma-t-3i+3}(t-\gamma+1)(\gamma-t+3i-3)h_{t-\gamma+1}h_{\gamma-t+3i-3}-\cdots\\-\frac{ih_{i}}{u_{y}}\sum_{\gamma=1}^{2t-2i+1}(t-\gamma+1)(\gamma-t+2i-1)h_{t-\gamma+1}h_{\gamma-t+2i-1}-\cdots\\
	 -\frac{(3i-t-2)h_{3i-t-2}}{u_{y}}\sum_{\gamma=1}^{t-1}(t-\gamma+1)(\gamma+1)h_{t-\gamma+1}h_{\gamma+1}\\-\frac{(3i-t-1)h_{3i-t-1}}{u_{y}}\sum_{\gamma=2}^{t-1}(t-\gamma+1)\gamma h_{t-\gamma+1}h_{\gamma}
	 -\cdots\\-\frac{th_{t}}{u_{y}}\sum_{\gamma=2t-3i+3}^{t-1}(t-\gamma+1)(\gamma-2t+3i-1)h_{t-\gamma+1}h_{\gamma-2t+3i-1}=0
\end{split}
\end{equation}
by substituting the equations in Lemma \ref{lem2.1}, Lemma \ref{lem2.2} and equation \eqref{eq2.11} to equation \eqref{eq2.38}. Since $h_1=h_2=\cdots=h_{i-1}=0$ by induction hypothesis, we have $-\frac{(ih_{i})^{3}}{u_{y}}=0$ by substituting them to equations \eqref{eq2.39}, \eqref{eq2.40}, respectively. Thus, we have $h_i=0$. Hence we have $h_1=h_2=\cdots=h_t=0$, which is a contradiction. Then the conclusion follows.
\end{proof}

\section{The case of $2\deg_zh\leq\deg_zv\leq 3\deg_zh-2$}

In the section, we classify polynomial maps of the form
$H=(u(x,y),v(x,y,z),\allowbreak h(x,y,z))$ in the case that $JH$ is
nilpotent and $2\deg_zh\leq\deg_zv\leq 3\deg_zh-2$.

\begin{prop}\label{prop3.1}
Let $H=(u(x,y),v(x,y,z),h(x,y,z))$ be a polynomial map of $K[x,y,z]$ with $H(0)= 0$. If $JH$ is nilpotent and $2\deg_zh\leq\deg_zv\leq 3\deg_zh-2$, then $u,~v,~h$ are linearly dependent over $K$.
\end{prop}
\begin{proof}
We prove it by contradiction. Suppose that $u,~v,~h$ are linearly independent over $K$. It follows from Proposition 2.2 in \cite{13} that $u_y\neq 0$. Let $v=\sum_{i=0}^dv_iz^i$, $h=\sum_{j=0}^th_jz^j$ with $v_dh_t\neq 0$ and $v_i, h_j\in K[x,y]$ for all $1\leq i\leq d$, $1\leq j\leq t$. Since $2\deg_zh\leq\deg_zv\leq 3\deg_zh-2$, we have $t\geq 2$. Since $JH$ is nilpotent, we have
\begin{eqnarray}
-u_{x}-v_{y}=h_{z},\label{eq3.1}\\
-u_{x}v_{y}-v_{x}u_{y}-u_{x}^{2}-v_{y}^{2}=h_{y}v_{z},\label{eq3.2}\\
-u_{x}^{3}-2u_{x}u_{y}v_{x}-v_{x}v_{y}u_{y}=u_{y}v_{z}h_{x}.\label{eq3.3}
\end{eqnarray}
It follows from equation \eqref{eq3.1} that
\begin{equation}\label{eq3.4}
-u_{x}-(v_{dy}z^{d}+\cdots+v_{1y}z+v_{0y})=th_{t}z^{t-1}+\cdots+2h_{2}z+h_{1}.
\end{equation}
Then we have $v_t,v_{t+1},\ldots,v_d\in K[x]$ and
\begin{eqnarray}
(v_{t-1})_y=-th_{t},\notag\\
	(v_{t-2})_y=-(t-1)h_{t-1},\notag\\
	\vdots~~~~~~~~~~\label{eq3.5}\\
	v_{1y}=-2h_{2},\notag\\
	-u_{x}-v_{0y}=h_{1}\notag
\end{eqnarray}
by comparing the coefficients of the degree of $z$ of equation \eqref{eq3.4}. It follows from equations \eqref{eq3.3}, \eqref{eq3.5} that
\begin{equation}\label{eq3.6}
\begin{split}
-u_{x}^{3}-2u_{x}u_{y}(v_{dx}z^{d}+\cdots+v_{0x})-u_{y}(v_{dx}z^{d}+\cdots+v_{0x})( (v_{t-1})_yz^{t-1}+\\\cdots+v_{0y})
	=u_{y}(h_{tx}z^{t}+\cdots+h_{0x})(dv_{d}z^{d-1}+\cdots+v_{1}).
\end{split}
\end{equation}
It follows from Lemma \ref{lem2.1} that $h_1,~h_2,\cdots,h_{t-1}\in K[x]$, $v_d,~h_t,~u_y\in K^*$. It's easy to see that $u_x\in K[x]$. Thus, we have
\begin{equation}\label{eq3.7}
(v_{2t-k})_x=\frac{(2t-k+1)v_{2t-k+1}v_{d-1}^{'}}{dv_{d}}+m_{k-2}(x),\,2\leq k\leq2t-1
\end{equation}
with $m_0(x)=-\frac{1}{u_y}(th_t)^2$. Let $d=3t-\beta-1$ with $1\leq\beta<t$. Next we prove that $h_{t-\alpha+1}^{'}=\frac{(t-\alpha+2)h_{t-\alpha+2}v_{d-1}^{'}}{dv_{d}}$ for all $2\leq\alpha\leq t-\beta+1$ by induction on $\alpha$.

If $\alpha=2$, then we have $h_{t-1}^{'}=\frac{th_{t}v_{d-1}^{'}}{dv_{d}}$ by following the arguments of Lemma \ref{lem2.2}. Suppose that
\begin{equation}\label{eq3.8}
h_{t-\kappa+1}^{'}=\dfrac{(t-\kappa+2)h_{t-\kappa+2}v_{d-1}^{'}}{dv_{d}}
\end{equation}
for all $2\leq\kappa\leq\alpha-1$. Then we have
\begin{equation}\nonumber
\begin{split}
dv_{d}(h_{t-\alpha+1})_x+(d-1)v_{d-1}(h_{t-\alpha+2})_x+\cdots+(d-\alpha+2)v_{d-\alpha+2}(h_{t-1})_x\\=-(v_{d-1})_x(v_{t-\alpha+1})_y-(v_{d-2})_x(v_{t-\alpha+2})_y-\cdots-(v_{d-\alpha+1})_x(v_{t-1})_y
\end{split}
\end{equation}
by comparing the coefficients of $z^{d+t-\alpha}$$(d+t-\alpha\geq d)$ of equation \eqref{eq3.6}. Substituting equation \eqref{eq3.5} to the above equation, we have
\begin{equation}\label{eq3.9}
\begin{split}
dv_{d}h_{t-\alpha+1}^{'}+(d-1)v_{d-1}h_{t-\alpha+2}^{'}+\cdots+(d-\alpha+2)v_{d-\alpha+2}h_{t-1}^{'}\\
=(t-\alpha+2)h_{t-\alpha+2}(v_{d-1})_x+(t-\alpha+3)h_{t-\alpha+3}(v_{d-2})_x+\cdots+th_{t}(v_{d-\alpha+1})_x.
\end{split}
\end{equation}
Since $d=3t-\beta-1$ and $2\leq\alpha\leq t-\beta+1$, we have $d-\alpha+1\geq 2t-1$. Substituting equations \eqref{eq2.18}, \eqref{eq3.8} to equation \eqref{eq3.9}, we have
\begin{equation}\label{eq3.10}
h_{t-\alpha+1}^{'}=\frac{(t-\alpha+2)h_{t-\alpha+2}v_{d-1}^{'}}{dv_{d}},\,2\leq\alpha\leq t-\beta+1.
\end{equation}
Comparing the coefficients of $z^{d+\beta-2}$$(d+\beta-2\geq d+1)$ of equation \eqref{eq3.6} and substituting equation \eqref{eq3.5} to it, we have
\begin{equation}\nonumber
\begin{split}
dv_{d}h_{\beta-1}^{'}+(d-1)v_{d-1}h_{\beta}^{'}+\cdots+(d+\beta-t)v_{d+\beta-t}h_{t-1}^{'}
\\=\beta h_{\beta}(v_{d-1})_x+(\beta+1)h_{\beta+1}(v_{d-2})_x+\cdots+th_{t}(v_{d+\beta-t-1})_x.
\end{split}
\end{equation}
Substituting equations \eqref{eq2.18}, \eqref{eq3.7}, \eqref{eq3.10} to the above equation, we have
\begin{equation}\label{eq3.11}
h_{\beta-1}^{'}=\frac{\beta h_{\beta}v_{d-1}^{'}}{dv_{d}}+q_{0}(x),\,q_{0}(x)=\frac{th_{t}m_{0}(x)}{dv_{d}}=-\frac{(th_t)^3}{dv_du_y}.
\end{equation}
Suppose that
\begin{equation}\label{eq3.12}
h_{\beta-\kappa}^{'}=\frac{(\beta-\kappa+1) h_{\beta-\kappa+1}v_{d-1}^{'}}{dv_{d}}+q_{\kappa-1}(x)
\end{equation}
for all $1\leq \kappa\leq j\;(1\leq j\leq\beta-2)$. Then we have
\begin{equation}\nonumber
\begin{split}
dv_{d}h_{\beta-j-1}^{'}+(d-1)v_{d-1}h_{\beta-j}^{'}+\cdots+(d-j)v_{d-j}h_{\beta-1}^{'}+\cdots\\
+(d+\beta-t-j)v_{d+\beta-t-j}h_{t-1}^{'}\\
=(\beta-j)h_{\beta-j}(v_{d-1})_x+(\beta-j+1)h_{\beta-j+1}(v_{d-2})_x+\cdots+(t-j)h_{t-j}(v_{d+\beta-t-1})_x\\
+\cdots+th_{t}(v_{d+\beta-t-j-1})_x
\end{split}
\end{equation}
by comparing the coefficients of $z^{d+\beta-j-2}$$(d+\beta-j-2\geq d+1)$ of equation \eqref{eq3.6} and substituting equation \eqref{eq3.5} to it. Since $d=3t-\beta-1$, we have $d+\beta-t-1=2t-2$. Substituting equations \eqref{eq2.18}, \eqref{eq3.7}, \eqref{eq3.10}, \eqref{eq3.12} to the above equation, we have
\begin{equation}\nonumber
h_{\beta-j-1}^{'}=\frac{(\beta-j)h_{\beta-j}v_{d-1}^{'}}{dv_{d}}+q_{j}(x)
\end{equation}
with $q_{j}(x)=
-\frac{1}{dv_{d}}[(d-1)v_{d-1}q_{j-1}(x)+\cdots+(d-j)v_{d-j}q_{0}(x)]+\frac{1}{dv_{d}}[(t-j)h_{t-j}m_{0}(x)+\cdots+th_{t}m_{j}(x)]$.
Clearly, we have $q_j(x)\in K[x]$. Thus, we have
\begin{equation}\label{eq3.13}
h_{\beta-j}^{'}=\frac{(\beta-j+1) h_{\beta-j+1}v_{d-1}^{'}}{dv_{d}}+q_{j-1}(x),\;q_{j-1}(x)\in K[x]
\end{equation}
for all $1\leq j\leq\beta-2$. It follows from equation \eqref{eq2.14} that
\begin{equation}\nonumber
\begin{split}
dv_{d}c_{1}^{'}(x)+(d-1)v_{d-1}h_{2}^{'}+\cdots+(d-\beta+2)v_{d-\beta+2}h_{\beta-1}^{'}+\cdots\\+(d-t+2)v_{d-t+2}h_{t-1}^{'}
=2h_{2}v_{d-1}^{'}+3h_3v_{d-2}^{'}+\cdots+(d-2t+3)h_{d-2t+3}v_{2t-2}^{'}\\+\cdots+th_{t}v_{d-t+1}^{'}.
\end{split}
\end{equation}
Substituting equations \eqref{eq2.18}, \eqref{eq3.7}, \eqref{eq3.10}, \eqref{eq3.13} to the above equation, we have
\begin{equation}\label{eq3.14}
c_{1}^{'}(x)=\frac{2h_{2}v_{d-1}^{'}}{dv_{d}}+q_{\beta-2}(x)
\end{equation}
with $q_{\beta-2}(x)=-\frac{1}{dv_{d}}[(d-1)v_{d-1}q_{\beta-3}(x)+\cdots+(d-\beta+2)v_{d-\beta+2}q_{0}(x)]+\frac{1}{dv_{d}}[(t-\beta+2)h_{t-\beta+2}m_{0}(x)+\cdots+th_{t}m_{\beta-2}(x)] $. It follows from equation \eqref{eq2.15} that
\begin{equation}\nonumber
\begin{split}
dv_{d}c_{2}^{'}(x)+(d-1)v_{d-1}c_{1}^{'}(x)+\cdots+(d-\beta+1)v_{d-\beta+1}h_{\beta-1}^{'}\\
+\cdots+(d-t+1)v_{d-t+1}h_{t-1}^{'}\\
=c_{1}(x)(v_{d-1})_x+2h_{2}(v_{d-2})_x+\cdots+(t-\beta+1)h_{t-\beta+1}(v_{d-t+\beta-1})_x+\cdots+th_{t}(v_{d-t})_x.
\end{split}
\end{equation}
Substituting equations \eqref{eq2.18}, \eqref{eq3.7}, \eqref{eq3.10}, \eqref{eq3.13}, \eqref{eq3.14} to the above equation, we have
\begin{equation}\label{eq3.15}
c_{2}^{'}(x)=\frac{c_{1}(x)v_{d-1}^{'}}{dv_{d}}+q_{\beta-1}(x)
\end{equation}
with $q_{\beta-1}(x)=-\frac{1}{dv_{d}}[(d-1)v_{d-1}q_{\beta-2}(x)+\cdots+(d-\beta+2)v_{d-\beta+2}q_{0}(x)]+\frac{1}{dv_{d}}[(t-\beta+1)h_{t-\beta+1}m_{0}(x)+\cdots+th_{t}m_{\beta-1}(x)]$.
Comparing the coefficients of $z^{t+\beta-3}$ of equation \eqref{eq3.6} and substituting equation \eqref{eq3.5} to it, we have
\begin{equation}\label{eq3.16}
\begin{split}
(t+\beta-2)v_{t+\beta-2}h_{0x}+(t+\beta-3)v_{t+\beta-3}h_{1x}+\cdots+(t-1)v_{t-1}h_{\beta-1}^{'}\\
+\cdots+(\beta-1)v_{\beta-1}h_{t-1}^{'}\\
=(h_{1}-u_{x})(v_{t+\beta-3})_x+2h_{2}(v_{t+\beta-4})_x+\cdots+th_{t}(v_{\beta-2})_x.
\end{split}
\end{equation}
Substituting equations \eqref{eq2.17}, \eqref{eq3.7}, \eqref{eq3.10}, \eqref{eq3.13}, \eqref{eq3.14}, \eqref{eq3.15} to equation \eqref{eq3.16}, we have
\begin{equation}\label{eq3.17}
\begin{split}
(t+\beta-2)v_{t+\beta-2}q_{\beta -1}(x)+(t+\beta-3)v_{t+\beta-3}q_{\beta -2}(x)+\cdots+(t-1)v_{t-1}q_{0}(x)\\
=(h_{1}-u_{x})m_{t-\beta+1}(x)+2h_{2}m_{t-\beta+2}(x)+\cdots+th_{t}m_{2t-\beta}(x).
\end{split}
\end{equation}
Since $v_t,~v_{t+1},\cdots,v_d\in K[x]$, $h_1,~h_2,\cdots,h_t,~u_x\in K[x]$, we have $-t(t-1)h_tq_0(x)=0$ by comparing the coefficients of $y$ of equation \eqref{eq3.17}. It follows from equation \eqref{eq3.11} that $-(t-1)\frac{(th_t)^4}{dv_du_y}=0$. Hence we have $h_t=0$, which is a contradiction. Thus, $u,~v,~h$ are linearly dependent over $K$.
\end{proof}

\begin{thm}\label{thm3.2}
Let $H=(u(x,y),v(x,y,z),h(x,y,z))$ be a polynomial map of $K[x,y,z]$ with $H(0)= 0$. Suppose that $u,~v,~h$ are linearly independent over $K$. If $JH$ is nilpotent and $\deg_zv\geq 2\deg_zh$, then $H$ has the form of Theorem \ref{thm2.3}
\end{thm}
\begin{proof}
If $\deg_zv\geq 3\deg_zh-1$, then the conclusion follows from Theorem \ref{thm2.3}. If $2\deg_zh\leq\deg_zv\leq 3\deg_zh-2$, then the conclusion follows from Proposition \ref{prop3.1}.
\end{proof}

\begin{thm}\label{thm3.3}
Let $H=(u(x,y),v(x,y,z),h(x,y,z))$ be a polynomial map of $K[x,y,z]$ with $H(0)= 0$. Suppose that $u,~v,~h$ are linearly independent over $K$. If $JH$ is nilpotent and $\deg_zh\leq 3$, then there exist $T\in \operatorname{GL}_3(K)$ such that $THT^{-1}$ has the form of Theorem \ref{thm2.3}
\end{thm}
\begin{proof}
If $\deg_zv\geq 6$, then the conclusion follows from Theorem \ref{thm3.2}. It follows from Lemma 3.1 and Lemma 3.2 in \cite{13} that we can assume $\deg_zv>\deg_zh$.

If $\deg_zh\leq 2$, then the conclusion follows from Theorem 3.4\footnote{The condition $(m,n)=1$ in Theorem 3.4 of \cite{13} can be removed due to Lemma 2.1 in \cite{14}.} in \cite{13}. Thus, we only need to check the following two cases: $\deg_zh=3$ and $\deg_zv=4$ or 5. Then the conclusion follows from Proposition \ref{prop3.5} and Proposition \ref{prop3.6}.

\begin{lem}\label{lem3.4}
Let $H=(u(x,y),v(x,y,z),h(x,y,z))$ be a polynomial map of $K[x,y,z]$ with $H(0)= 0$. Suppose that  $v=\sum_{i=0}^dv_iz^i$, $h=\sum_{j=0}^3h_jz^j$ with $v_dh_3\neq 0$ and $v_i, h_j\in K[x,y]$. If $JH$ is nilpotent and $\deg_zh= 3$, $\deg_zv=4$ or 5, then $h_2,~h_3\in K[x]$ and $v_{2y}=-3h_3$, $v_{1y}=-2h_2$, $-u_x-v_{0y}=h_1$.
\end{lem}

Since $JH$ is nilpotent, we have
\begin{eqnarray}
-u_{x}-v_{y}=h_{z},\label{eq3.18}\\
-u_{x}v_{y}-v_{x}u_{y}-u_{x}^{2}-v_{y}^{2}=h_{y}v_{z},\label{eq3.19}\\
-u_{x}^{3}-2u_{x}u_{y}v_{x}-v_{x}v_{y}u_{y}=u_{y}v_{z}h_{x}.\label{eq3.20}
\end{eqnarray}
It follows from equation \eqref{eq3.18} that
\begin{equation}\label{eq3.21}
-u_{x}-(v_{dy}z^{d}+\cdots+v_{1y}z+v_{0y})=3h_{3}z^2+2h_{2}z+h_{1}.
\end{equation}
Then we have $v_3,v_{4},\ldots,v_d\in K[x]$ and
\begin{eqnarray}
v_{2y}=-3h_{3},\notag\\
	v_{1y}=-2h_{2},\label{eq3.22}\\
	-u_{x}-v_{0y}=h_{1}\notag
\end{eqnarray}
by comparing the coefficients of the degree of $z$ of equation \eqref{eq3.21}.
It follows from equation \eqref{eq3.19} that
\begin{equation}\label{eq3.23}
\begin{split}
-u_{x}\left(v_{2y}z^{2}+v_{1y}z+v_{0y}\right)-u_{y}\left( v_{dx}z^{d}+\cdots+v_{0x}\right) -u_{x}^{2}\\-\left( v_{2y}z^{2}+v_{1y}z+v_{0y}\right) ^{2}\\=\left( h_{3y}z^{3}+\cdots+h_{0y}\right) \left( dv_{d}z^{d-1}+\cdots+v_{1}\right).
\end{split}
\end{equation}
Then we have
\begin{equation}\label{eq3.24}
h_3,~h_2\in K[x]
\end{equation}
by comparing the coefficients of $z^{d+2}$ and $z^{d+1}$ of equation \eqref{eq3.23}, respectively. Then the conclusion follows.

\begin{prop}\label{prop3.5}
Let $H=(u(x,y),v(x,y,z),h(x,y,z))$ be a polynomial map of $K[x,y,z]$ with $H(0)= 0$. Suppose that  $v=\sum_{i=0}^dv_iz^i$, $h=\sum_{j=0}^3h_jz^j$ with $v_dh_3\neq 0$ and $v_i, h_j\in K[x,y]$. If $JH$ is nilpotent and $d=5$ , then $u,~v,~h$ are linearly dependent over $K$.
\end{prop}

Suppose that $u,~v,~h$ are linearly independent over $K$. Then it follows from Proposition 2.2 in \cite{13} that we have $u_y\neq 0$. Comparing the coefficients of $z^5$ of equation \eqref{eq3.23}, we have
$$h_{1y}=-\frac{v_5^{'}}{5v_5}u_y.$$
Then we have
\begin{equation}\label{eq3.25}
h_1=-\frac{v_5^{'}}{5v_5}u+c_1(x),~c_1(x)\in K[x]
\end{equation}
by integrating the two sides with respect to $y$ of the above equation. Comparing the coefficients of $z^4$ of equation \eqref{eq3.23}, we have
$$h_{0y}=\left(\frac{4v_4v_5^{'}}{(5v_5)^2}-\frac{v_4^{'}}{5v_5}\right)u_y-\frac{9h_3^2}{5v_5}.$$
Then we have
\begin{equation}\label{eq3.26}
h_0=\left(\frac{4v_4v_5^{'}}{(5v_5)^2}-\frac{v_4^{'}}{5v_5}\right)u-\frac{9h_3^2}{5v_5}y+c_2(x),~c_2(x)\in K[x]
\end{equation}
by integrating the two sides with respect to $y$ of the above equation. It follows from equations \eqref{eq3.20}, \eqref{eq3.22} that
\begin{equation}\label{eq3.27}
\begin{split}
-u_{x}^{3}-2u_{x}u_{y}\left(v_{5x}z^{5}+\cdots+v_{0x}\right)-u_{y}\left(v_{5x}z^{5}+\cdots+v_{0x}\right) \left( v_{2y}z^{2}+v_{1y}z+v_{0y}\right)\\
	=u_{y}\left( h_{3x}z^{3}+\cdots+h_{0x}\right) \left( 5v_{5}z^{4}+\cdots+v_{1}\right).
\end{split}
\end{equation}
Comparing the coefficients of the highest degree of $z$ of equation \eqref{eq3.27}, we have
$$5v_5h_{3x}=3h_3v_{5x}.$$
Then we have
\begin{equation}\label{eq3.28}
v_5^3=\lambda h_3^5,~\lambda\in K^*
\end{equation}
by integrating the two sides with respect to $x$ of the above equation. Comparing the coefficients of $z^5$ of equation \eqref{eq3.27} and substituting equation \eqref{eq3.22} to it, we have
$$(h_1-u_x)v_5^{'}+2h_2v_4^{'}+3h_3v_3^{'}=5v_5h_{1x}+4v_4h_{2x}+3v_3h_{3x}.$$
Substituting equations \eqref{eq3.22}, \eqref{eq3.25} to the above equation, we have $v_5\in K^*$. It follows from equation \eqref{eq3.28} that $h_3\in K^*$. Thus, we have
\begin{eqnarray}
h_1=c_1(x),\label{eq3.29}\\
h_0=-\frac{v_4^{'}}{5v_5}u-\frac{9h_3^2}{5v_5}y+c_2(x).\label{eq3.30}
\end{eqnarray}
If $u_y=0$, then it follows from Proposition 2.1 in \cite{13} that $u,~v,~h$ are linearly dependent over $K$. Thus, we have $u_y\neq 0$. Since $v_d\in K^*$, $h_t\in K^*$, we have
\begin{equation}\label{eq3.31}
h_{2}^{'}=\frac{3h_{3}v_{4}^{'}}{5v_{5}}
\end{equation}
by comparing the coefficients of $z^6$ of equation \eqref{eq3.27}. Comparing the coefficients of $z^4$ of equation \eqref{eq3.27} and substituting equations \eqref{eq3.29}, \eqref{eq3.30} to it, we have
\begin{equation}\nonumber
\begin{split}
-v_{4}^{''}u+5v_{5}c_{2}^{'}(x)+4v_{4}c_{1}^{'}(x)+3v_{3}h_{2}^{'}=c_{1}(x)v_{4x}+2h_{2}v_{3x}+3h_{3}v_{2x}.
\end{split}
\end{equation}
Then we have $v_{4}^{''}=0$ by comparing the coefficients of $y$ of the above equation. That is, $v_4^{'}\in K$. Then the above equation has the following form:
\begin{equation}\label{eq3.32}
5v_{5}c_{2}^{'}(x)+4v_{4}c_{1}^{'}(x)+3v_{3}h_{2}^{'}=c_{1}(x)v_{4x}+2h_{2}v_{3x}+3h_{3}v_{2x}.
\end{equation}
Comparing the coefficients of $z^3$ of equation \eqref{eq3.23} and substituting equation \eqref{eq3.24} to it, we have
$$4v_{4}h_{0y}+3v_{3}h_{1y}=-v_{3x}u_{y}-2v_{2y}v_{1y}.$$
Substituting equations \eqref{eq3.22}, \eqref{eq3.29}, \eqref{eq3.30} to the above equation, we have
\begin{equation}\label{eq3.33}
\left(v_{3x}-\frac{4v_{4}v_{4}^{'}}{5v_{5}} \right) u_{y}=\frac{4v_{4}(3h_{3})^{2}}{5v_{5}}-12h_{3}h_{2}.
\end{equation}
Let $p(x)=\frac{4v_{4}(3h_{3})^{2}}{5v_{5}}-12h_{3}h_{2}$. It follows from equation \eqref{eq3.31} that $ p^{'}(x)=0$. That is, $p(x)\in K$.\\

$(1)$ If $u_y\in K^*$, then it follows from  equation \eqref{eq3.33} that
\begin{equation}\label{eq3.34}
v_{3x}=\frac{4v_4v_4^{'}}{5v_5}+m_1(x)
\end{equation}
with $m_1(x)\in K[x]$. Comparing the coefficients of $z^2$ of equation \eqref{eq3.23}, we have
$$3v_3h_{0y}=-u_xv_{2y}-u_yv_{2x}-2v_{2y}v_{0y}-v_{1y}^2.$$
Substituting equations \eqref{eq3.22}, \eqref{eq3.30} to the above equation, we have
\begin{equation}\label{eq3.35}
v_{2x}=\frac{3v_3v_4^{'}}{5v_5}+m_2(x)
\end{equation}
with $m_2(x)\in K[x]$. Then we have
\begin{equation}\label{eq3.36}
v_{1x}=\frac{2v_2v_4^{'}}{5v_5}+\frac{2v_2(3h_3)^2}{5v_5u_y}+m_3(x)
\end{equation}
by comparing the coefficients of $z$ of equation \eqref{eq3.23} and substituting equations \eqref{eq3.22}, \eqref{eq3.30} to it, where $m_3(x)\in K[x]$. Comparing the coefficients of $z^5$ of equation \eqref{eq3.27}, we have
$$5v_{5}h_{1}^{'}+4v_{4}h_{2}^{'}=2h_{2}v_{4}^{'}+3h_{3}v_{3}^{'}.$$
Then we have
\begin{equation}\label{eq3.37}
h_{1}^{'}=\frac{2h_{2}v_{4}^{'}}{5v_{5}}+q_{1}(x),\;q_{1}(x)=\frac{3h_{3}m_{1}(x)}{5v_{5}}
\end{equation}
by substituting equations \eqref{eq3.31}, \eqref{eq3.34} to the above equation. It follows from equations \eqref{eq3.32}, \eqref{eq3.34}, \eqref{eq3.35}, \eqref{eq3.37} that
\begin{equation}\label{eq3.38}
c_{2}^{'}(x)=\frac{h_1v_{4}^{'}}{5v_{5}}+q_{2}(x),
\end{equation}
where $q_{2}(x)=-\frac{4v_4}{5v_5}q_1(x)+\frac{2h_{2}m_{1}(x)+3h_3m_2(x)}{5v_{5}}$. Comparing the coefficients of $z^3$ of equation \eqref{eq3.27}, we have
\begin{equation}\label{eq3.39}
4v_{4}h_{0x}+3v_{3}h_{1}^{'}+2v_{2}h_{2}^{'}=(h_{1}-u_{x})v_{3x}+2h_{2}v_{2x}+3h_{3}v_{1x}.
\end{equation}
Substituting equations \eqref{eq3.30}, \eqref{eq3.31}, \eqref{eq3.37}, \eqref{eq3.38}, \eqref{eq3.34}, \eqref{eq3.35}, \eqref{eq3.36} to the above equation, we have
$$4v_4q_2(x)+3v_3q_1(x)=(h_1-u_x)m_1(x)+2h_2m_2(x)+3h_3m_3(x)+\frac{2v_2(3h_3)^3}{5v_5u_y}.$$
Since $h_1,~h_2,~h_3,~u_x \in K[x]$, $v_3,~v_4\in K[x]$, we have $2\frac{(3h_3)^4}{5v_5u_y}=0$ by differentiating both sides of the above equation with respect to $y$. Hence we have $h_t=0$, which is a contradiction. \\

$(2)$ If $u_y\in K[x]\backslash K$, then we have $u=l_1(x)y+l_0(x)$ with $l_0(x)\in K[x]$, $l_1(x)\in K[x]\backslash K$. Then we have
\begin{equation}\label{eq3.40}
3v_{3}h_{0y}=-v_{2x}u_{y}-v_{2y}u_{x}-2v_{2y}v_{0y}-v_{1y}^2
\end{equation}
by comparing the coefficients of $z^2$ of equation \eqref{eq3.23}. It follows from equation \eqref{eq3.22} that
\begin{eqnarray}
v_2=-3h_3y+k_1(x),~k_1(x)\in K[x],\label{eq3.41}\\
v_1=-2h_2y+k_2(x),~k_2(x)\in K[x].\label{eq3.42}
\end{eqnarray}
Substituting equations \eqref{eq3.22}, \eqref{eq3.30}, \eqref{eq3.41}, \eqref{eq3.42} to equation \eqref{eq3.40}, we have
\begin{equation}\label{eq3.43}
u_{x}=-\frac{1}{3h_{3}}\left( k_{1}^{'}(x)-\frac{3v_{3}v_{4}^{'}}{5v_{5}}\right) u_{y}+\frac{3v_{3}3h_{3}}{5v_{5}}-\frac{4h_2^2}{3h_{3}}-2h_{1}
\end{equation}
Then we have $l_{1}^{'}(x)=0$ by comparing the coefficients of $y$ of equation \eqref{eq3.43}. That is, $l_1(x)\in K$, which contradicts with the condition $l_1(x)\in K[x]\backslash K$.\\

$(3)$ If $\deg_yu_y\geq 1$, then it follows from equation \eqref{eq3.33} that
\begin{eqnarray}\label{eq3.44}
v_{3x}=\frac{4v_4v_4^{'}}{5v_5}\\\nonumber
h_2=\frac{3v_4h_3}{5v_5}
\end{eqnarray}
Comparing the coefficients of $z^2$ of equation \eqref{eq3.23}, we have
$$3v_{3}h_{0y}=-v_{2x}u_{y}-v_{2y}u_x-2v_{2y}v_{0y}-v_{1y}^2.$$
Substituting equations \eqref{eq3.23}, \eqref{eq3.30}, \eqref{eq3.41} to the above equation, we have
\begin{equation}\label{eq3.45}
v_{0y}=\frac{1}{3h_{3}}\left( k_{1}^{'}(x)-\frac{3v_{3}v_{4}^{'}}{5v_{5}}\right) u_{y}-\frac{3v_{3}3h_{3}}{5v_{5}}+\frac{4h_2^2}{3h_{3}}+h_1.
\end{equation}
Comparing the coefficients of $z^5$ of equation \eqref{eq3.27} and substituting equation \eqref{eq3.22} to it, we have
$$2h_2v_4^{'}+3h_3v_3^{'}=5v_5h_1^{'}+4v_4h_2^{'}.$$
Substituting equations \eqref{eq3.31}, \eqref{eq3.44} to the above equation, we have
\begin{equation}\label{eq3.46}
h_1^{'}=\frac{2h_2v_4^{'}}{5v_5}.
\end{equation}
It follows from equation \eqref{eq3.32} that
$$5v_5c_2^{'}(x)+4v_4h_1^{'}+3v_3h_2^{'}=h_1v_4^{'}+2h_2v_3^{'}+3h_3v_2^{'}.$$
Substituting equations \eqref{eq3.31}, \eqref{eq3.34}, \eqref{eq3.41}, \eqref{eq3.46} to the above equation, we have
\begin{equation}\label{eq3.47}
c_2^{'}(x)=\frac{h_1v_4^{'}}{5v_5}-\frac{9v_3h_3v_4^{'}}{(5v_5)^2}+\frac{3h_3k_1^{'}(x)}{5v_5}.
\end{equation}
Comparing the coefficients of $z^3$ of equation \eqref{eq3.27} and substituting equation \eqref{eq3.22} to it, we have
$$4v_{4}h_{0x}+3v_{3}h_{1}^{'}+2v_{2}h_{2}^{'}=(h_{1}-u_{x})v_{3x}+2h_{2}v_{2x}+3h_{3}v_{1x}.$$
Substituting equations \eqref{eq3.31}, \eqref{eq3.41}, \eqref{eq3.42}, \eqref{eq3.44}, \eqref{eq3.46}, \eqref{eq3.47} to the above equation, we have
\begin{equation}\label{eq3.48}
k_{2}^{'}(x)=-\frac{6v_{4}v_{3}v_{4}^{'}}{(5v_5)^{2}}+\frac{2v_{4}^{'} k_{1}(x)}{5v_{5}}+\frac{2v_{4}k_{1}^{'}(x)}{5v_{5}}.
\end{equation}
Then we have
\begin{equation}\nonumber
2v_{2}h_{0y}=-u_{y}v_{1x}-2h_{2}(h_{1}-v_{0y})
\end{equation}
by comparing the coefficients of $z$ of equation \eqref{eq3.23} and substituting equation \eqref{eq3.22} to it. Substituting equations \eqref{eq3.30}, \eqref{eq3.41}, \eqref{eq3.42}, \eqref{eq3.44}, \eqref{eq3.45}, \eqref{eq3.48} to the above equation, we have
$$\frac{2(3h_3)^3}{5v_5}y-\frac{2(3h_3)^2}{5v_5}k_1(x)+\frac{18h_2h_3v_3}{5v_5}-\frac{(2h_2)^3}{3h_3}=0.$$
Then we have $h_3=0$ by comparing the coefficients of $y$ of the above equation, which is a contradiction. Thus, $u,~v,~h$ are linearly dependent over $K$.

\begin{prop}\label{prop3.6}
Let $H=(u(x,y),v(x,y,z),h(x,y,z))$ be a polynomial map of $K[x,y,z]$ with $H(0)= 0$. Suppose that  $v=\sum_{i=0}^dv_iz^i$, $h=\sum_{j=0}^3h_jz^j$ with $v_dh_3\neq 0$ and $v_i, h_j\in K[x,y]$. If $JH$ is nilpotent and $d=4$ , then $u,~v,~h$ are linearly dependent over $K$.
\end{prop}

Suppose that $u,~v,~h$ are linearly independent over $K$. Then it follows from Proposition 2.2 in \cite{13} that we have $u_y\neq 0$.
Thus, we have
$$4v_4h_{1y}=-v_{4x}u_y-v_{2y}^2$$
by comparing the coefficients of $z^4$ of equation \eqref{eq3.23}. Substituting equation \eqref{eq3.22} to the above equation, we have
$$h_{1y}=-\frac{v_4^{'}}{4v_4}u_y-\frac{(3h_3)^2}{4v_4}.$$
Then we have
\begin{equation}\label{eq3.49}
h_1=-\frac{v_4^{'}}{4v_4}u-\frac{(3h_3)^2}{4v_4}y+c_1(x),~c_1(x)\in K[x]
\end{equation}
by integrating the two sides of the above equation with respect to $y$. Comparing the coefficients of $z^3$ of equation \eqref{eq3.23}, we have
$$4v_4h_{0y}+3v_3h_{1y}=-v_{3x}u_y-2v_{2y}v_{1y}.$$
Substituting equations \eqref{eq3.22}, \eqref{eq3.49} to the above equation, we have
$$h_{0y}=\left( \frac{3v_3v_{4}^{'}}{(4v_{4})^{2}}-\frac{v_3^{'}}{4v_{4}}\right)u_{y}+\frac{3v_3(3h_{3})^{2}}{(4v_{4})^{2}}-\frac{12h_{3}h_{2}}{4v_{4}} .$$
Then we have
\begin{equation}\label{eq3.50}
h_{0}=\left( \frac{3v_3v_{4}^{'}}{(4v_{4})^{2}}-\frac{v_3^{'}}{4v_{4}}\right)u+\left(\frac{3v_3(3h_{3})^{2}}{(4v_{4})^{2}}-\frac{12h_{3}h_{2}}{4v_{4}}\right)y +c_2(x),~c_2(x)\in K[x]
\end{equation}
by integrating the two sides of the above equation with respect to $y$. It follows from equation \eqref{eq3.20} that
\begin{equation}\label{eq3.51}
\begin{split}
-u_{x}^{3}-2u_{x}u_{y}\left(v_{4x}z^{4}+\cdots+v_{0x}\right)-u_{y}\left(v_{4x}z^{4}+\cdots+v_{0x}\right) \left( v_{2y}z^{2}+v_{1y}z+v_{0y}\right)\\=u_{y}\left( h_{3x}z^{3}+\cdots+h_{0x}\right) \left( 4v_{4}z^{3}+\cdots+v_{1}\right).
\end{split}
\end{equation}
Then we have
$$3h_3v_{4x}=4v_4h_{3x}$$
by comparing the coefficients of $z^6$ of equation (3.51). Then we have
\begin{equation}\label{eq3.52}
v_4^3=\lambda h_3^4,~\lambda\in K^*
\end{equation}
by integrating the two sides of the above equation with respect to $x$. It follows from equations \eqref{eq3.22}, \eqref{eq3.24} that
\begin{eqnarray}\label{eq3.53}
v_2=-3h_3y+k_2(x)\\\nonumber
v_{1}=-2h_{2}y+k_{1}(x).
\end{eqnarray}
Then we have
$$4v_4h_{1x}+3v_3h_{2x}+2v_2h_{3x}=(h_1-u_x)v_{4x}+2h_2v_{3x}+3h_3v_{2x}$$
by comparing the coefficients of $z^4$ of equation (3.51) and substituting equation \eqref{eq3.22} to it. Substituting equations \eqref{eq3.49}, \eqref{eq3.52}, \eqref{eq3.53} to the above equation, we have
$$\frac{5(v_{4}^{'})^2-4v_{4}v_{4}^{''}}{4v_{4}}u+4v_{4}c_{1}^{'}(x)+3v_3h_{2}^{'}+2k_{2}(x)h_{3}^{'}=c_{1}(x)v_{4}^{'}+2h_{2}v_{3}^{'}+3h_{3}k_{2}^{'}(x).$$
Since $u_y\neq 0$, we have $5(v_{4}^{'})^2=4v_{4}v_{4}^{''}$ by comparing the coefficients of the degree of $y$ of the above equation. Thus, we have $v_4\in K^*$. It follows from equation \eqref{eq3.52} that $h_3\in K^*$. Then the above equation has the following form:
\begin{equation}\label{eq3.54}
4v_{4}c_{1}^{'}(x)+3v_3h_{2}^{'}=2h_{2}v_{3}^{'}+3h_{3}k_{2}^{'}(x).
\end{equation}
Since $v_4,~h_3\in K^*$, then equations \eqref{eq3.49}, \eqref{eq3.50} have the following form:
\begin{eqnarray}
h_1=-\frac{(3h_3)^2}{4v_4}y+c_1(x),\label{eq3.55}\\
h_{0}=-\frac{v_3^{'}}{4v_{4}}u+\left(\frac{3v_3(3h_{3})^{2}}{(4v_{4})^{2}}-\frac{12h_{3}h_{2}}{4v_{4}}\right)y +c_2(x).\label{eq3.56}
\end{eqnarray}
Comparing the coefficients of $z^5$ of equation (3.51), we have
\begin{equation}\label{eq3.57}
h_{2}^{'}=\frac{3h_{3}v_{3}^{'}}{4v_{4}}.
\end{equation}
Then we have
$$4v_4h_{0x}+3v_3h_{1x}+2v_2h_{2x}=(h_1-u_x)v_{3x}+2h_2v_{2x}+3h_3v_{1x}$$
by comparing the coefficients of $z^3$ of equation (3.51) and substituting equation \eqref{eq3.22} to it. Substituting equations \eqref{eq3.53}, \eqref{eq3.55}, \eqref{eq3.56} to the above equation, we have
$$-v_{3}^{''}u+4v_{4}c_{2}^{'}(x)+3v_3c_{1}^{'}(x)+2k_{2}(x)h_{2}^{'}=c_{1}(x)v_{3}^{'}+2h_{2}k_{2}^{'}(x)+3h_{3}k_{1}^{'}(x).$$
Then we have $v_{3}^{''}=0$ by comparing the coefficients of the degree of $y$. That is, $v_{3}^{'}\in K$. Then the above equation has the following form:
\begin{equation}\label{eq3.58}
4v_{4}c_{2}^{'}(x)+3v_3c_{1}^{'}(x)+2k_{2}(x)h_{2}^{'}=c_{1}(x)v_{3}^{'}+2h_{2}k_{2}^{'}(x)+3h_{3}k_{1}^{'}(x).
\end{equation}
Substituting equation \eqref{eq3.57} to equation \eqref{eq3.54}, we have
\begin{equation}\label{eq3.59}
c_{1}^{'}(x)=\frac{2h_{2}v_{3}^{'}}{4v_{4}}-\frac{3v_{3}3h_{3}v_{3}^{'}}{(4v_{4})^{2}}+\frac{3h_{3}k_{2}^{'}(x)}{4v_{4}}.
\end{equation}
Substituting equations \eqref{eq3.57}, \eqref{eq3.59} to equation \eqref{eq3.58}, we have
\begin{equation}\label{eq3.60}
\begin{split}
c_{2}^{'}(x)=\frac{c_{1}(x)v_{3}^{'}}{4v_{4}}+\frac{(3v_{3})^23h_{3}v_{3}^{'}}{(4v_{4})^{3}}-\frac{3v_{3}3h_{3}k_{2}^{'}(x)}{(4v_{4})^{2}}+\frac{3h_{3}k_{1}^{'}(x)}{4v_{4}}\\
-\frac{3v_{3}2h_{2}v_{3}^{'}}{(4v_{4})^{2}}-\frac{2k_{2}(x)3h_{3}v_{3}^{'}}{(4v_{4})^{2}}+\frac{2h_{2}k_{2}^{'}(x)}{4v_{4}}.
\end{split}
\end{equation}
Comparing the coefficients of $z^2$ of equation \eqref{eq3.23}, we have
$$3v_3h_{0y}+2v_2h_{1y}=-v_{2x}u_y-v_{1y}^2+v_{2y}(-u_x-2v_{0y}).$$
Substituting equations \eqref{eq3.22}, \eqref{eq3.53}, \eqref{eq3.55}, \eqref{eq3.56} to the above equation, we have
$$v_{0y}=\frac{\theta_1(x)}{3h_3}u_y+\frac{(3h_3)^2}{4v_4}y+\pi_1(x).$$
Then we have
\begin{equation}\label{eq3.61}
v_{0}=\frac{\theta_1(x)}{3h_3}u+\frac{(3h_3)^2}{8v_4}y^2+\pi_1(x)y+k_0(x)
\end{equation}
by integrating the two sides of the above equation with respect to $y$,
where $\theta_1(x)=k_2^{'}(x)-\frac{3v_3v_3^{'}}{4v_4}$, $\pi_1(x)=3v_3\left( \frac{9v_3h_{3}}{(4v_{4})^{2}}-\frac{4h_{2}}{4v_{4}}\right)-\frac{6h_{3}k_{2}(x)}{4v_{4}}+c_{1}(x)+\frac{4h_2^2}{3h_{3}}$.
Substituting equations \eqref{eq3.55}, \eqref{eq3.61} to the third equation of equation\eqref{eq3.22}, we have
\begin{equation}\label{eq3.62}
u_x=-\frac{\theta_{1}(x)}{3h_{3}}u_{y}+\delta_{1}(x),
\end{equation}
where $\delta_1(x)=-c_1(x)-\pi_1(x)$. Then we have
$$3v_3h_{0x}+2v_2h_{1x}+v_1h_{2x}=(h_1-u_x)v_{2x}+2h_2v_{1x}+3h_3v_{0x}$$
by comparing the coefficients of $z^2$ of equation (3.51) and substituting equation \eqref{eq3.22} to it. Substituting equations \eqref{eq3.53}, \eqref{eq3.55}, \eqref{eq3.56}, \eqref{eq3.61}, \eqref{eq3.62} to the above equation, we have
\begin{equation}\label{eq3.63}
\begin{split}
\theta_{1}^{'}(x)u+\phi(x)y+c_{1}(x)k_{2}^{'}(x)+2h_{2}k_{1}^{'}(x)+3h_{3}k_{0}^{'}(x)\\
=3v_{3}c_{2}^{'}(x)+2k_{2}(x)c_{1}^{'}(x)+k_{1}(x)h_{2}^{'},
\end{split}
\end{equation}
where $\phi(x)=-\frac{(3h_3)^2}{4v_4}k_2^{'}(x)-2h_2h_2^{'}+3h_3\pi_1^{'}(x)-\frac{9v_3v_3^{'}(3h_3)^2}{(4v_4)^2}+\frac{36v_3h_3h_2^{'}}{4v_4}+6h_3c_1^{'}(x)$.
Substituting equations \eqref{eq3.57}, \eqref{eq3.59} and the equation of $\pi_1(x)$ to the equation $\phi(x)$, we have $\phi(x)=0$. Thus, we have $\theta_{1}^{'}(x)=0$ by comparing the coefficients of the degree of $y$ of equation \eqref{eq3.63}. Then equation \eqref{eq3.63} has the following form:
\begin{equation}\label{eq3.64}
\begin{split}
c_{1}(x)k_{2}^{'}(x)+2h_{2}k_{1}^{'}(x)+3h_{3}k_{0}^{'}(x)
=3v_{3}c_{2}^{'}(x)+2k_{2}(x)c_{1}^{'}(x)+k_{1}(x)h_{2}^{'}.
\end{split}
\end{equation}
Comparing the coefficients of $z^0$ of equation \eqref{eq3.23} and substituting equation \eqref{eq3.22} to it, we have
$$v_1h_{0y}=h_1v_{0y}-u_yv_{0x}-u_x^2.$$
Substituting equations \eqref{eq3.53}, \eqref{eq3.55}, \eqref{eq3.56}, \eqref{eq3.61}, \eqref{eq3.62} to the above equation, we have
\begin{equation}\label{eq3.65}
\begin{split}
\left[\frac{2h_2v_3^{'}}{4v_4}+\frac{3h_3\theta_1}{4v_4}+\pi_1^{'}(x)\right]yu_y+\left[-\frac{v_3^{'}k_1(x)}{4v_4}-\frac{\theta_1c_1(x)}{3h_3}+k_0^{'}(x)-\frac{\theta_1\delta_1(x)}{3h_3}\right]u_y\\
+\frac{(3h_3)^2}{(4v_4)^2}y^2+\eta(x)y+\frac{27v_3h_3^2}{(4v_4)^2}k_1(x)-\frac{12h_3h_2}{4v_4}k_1(x)-c_1(x)\pi_1(x)+\delta_1^2(x)=0,
\end{split}
\end{equation}
where $\eta(x)=-\frac{54h_2v_3h_3^2}{(4v_4)^2}+\frac{24h_2^2h_3}{4v_4}+\frac{(3h_3)^2\pi_1(x)}{4v_4}-\frac{(3h_3)^2c_1(x)}{4v_4}$.

$(i)$ If $\deg_yu\geq 4$ or $\deg_yu=1$, then we have $\frac{(3h_3)^2}{(4v_4)^2}=0$ by comparing the coefficients of the degree of $y$ of equation \eqref{eq3.65}. Then we have $h_3=0$, which is a contradiction.

$(ii)$ If $\deg_yu=2$, then we have $\frac{2h_2v_3^{'}}{4v_4}+\frac{3h_3\theta_1}{4v_4}+\pi_1^{'}(x)=0$ by substituting the equations of $\theta_1(x)$ and $\pi_1(x)$ to the coefficients of $yu_y$ of equation \eqref{eq3.65}. Then we have $\frac{(3h_3)^2}{(4v_4)^2}=0$ by comparing the coefficients of $y^2$ of equation \eqref{eq3.65}. Then we have $h_3=0$, which is a contradiction.

$(iii)$  If $\deg_yu=3$, then we have
$$2v_2h_{0y}+v_1h_{1y}=-v_{1x}u_y-v_{1y}u_x-2v_{1y}v_{0y}$$
by comparing the coefficients of $z$ of equation \eqref{eq3.23}. Substituting equations \eqref{eq3.22}, \eqref{eq3.55}, \eqref{eq3.61}, \eqref{eq3.53}, \eqref{eq3.56}, \eqref{eq3.57} to the above equation, we have
\begin{equation}\nonumber
\begin{split}
\left[k_1^{'}(x)-\frac{2k_2(x)v_3^{'}}{4v_4}-\frac{2h_2\theta_1}{3h_3}\right]u_y+\left(\frac{54h_3^2h_2}{4v_4}-\frac{6v_3(3h_3)^3}{(4v_4)^2}\right)y\\
+2\left(\frac{27v_3h_3^2}{(4v_4)^2}-\frac{12h_3h_2}{4v_4}\right)k_2(x)-\frac{9h_3^2}{4v_4}k_1(x)+2h_2c_1(x)-2h_2\pi_1(x)=0.
\end{split}
\end{equation}
Then we have
\begin{equation}\label{eq3.66}
k_1^{'}(x)=\frac{2k_2(x)v_3^{'}}{4v_4}+\frac{2h_2\theta_1}{3h_3}
\end{equation}
by comparing the highest degree of $y$ of the above equation. Substituting equations \eqref{eq3.59}, \eqref{eq3.60}, \eqref{eq3.66} to equation \eqref{eq3.64}, we have
\begin{equation}\label{eq3.67}
\begin{split}
k_{0}^{'}(x)=\frac{3v_{3}c_{1}(x)v_{3}^{'}}{12v_{4}h_{3}}-\frac{c_{1}(x)k_{2}^{'}(x)}{3h_{3}}+\frac{2k_{2}(x)k_{2}^{'}(x)}{4v_{4}}-\frac{6k_{2}(x)v_{3}v_{3}^{'}}{(4v_{4})^{2}}
-\frac{4\theta_{1}h_2^2}{(3h_{3})^{2}}\\+\frac{k_{1}(x)v_{3}^{'}}{4v_{4}}+\frac{(3v_{3})^3v_{3}^{'}}{(4v_{4})^{3}}
-\frac{(3v_{3})^2k_{2}^{'}(x)}{(4v_{4})^{2}}-\frac{12h_{2}v_{3}^{2}v_{3}^{'}}{(4v_{4})^{2}h_{3}}+\frac{v_{3}h_{2}k_{2}^{'}(x)}{v_{4}h_{3}}.
\end{split}
\end{equation}
Substituting equation \eqref{eq3.67} and the equations of $\theta_1$, $\delta_1$ to the coefficients of $u_y$ of equation \eqref{eq3.65}, we have $-\frac{v_3^{'}k_1(x)}{4v_4}-\frac{\theta_1c_1(x)}{3h_3}+k_0^{'}(x)-\frac{\theta_1\delta_1(x)}{3h_3}=0$.
Then we have $\frac{(3h_3)^2}{(4v_4)^2}=0$ by comparing the coefficients of the degree of $y$ of equation \eqref{eq3.65}. Thus, we have $h_3=0$, which is a contradiction. Therefore, $u,~v,~h$ are linearly dependent over $K$.
\end{proof}

\section{Problems for further research}

In the paper, we study the polynomial maps of the form: $H=(u(x,y),v(x,y,z),\allowbreak h(x,y,z))$. Next we want to study the case $\deg_zh+1\leq\deg_zv\leq 2\deg_zh-1$, and classify all polynomial maps in dimension three with nilpotent Jacobians. Thus, we can ask the following problem.

\begin{prob}\label{prob4.1}
Let $H=(u(x,y,z),v(x,y,z),h(x,y,z))$ be a polynomial map of $K[x,y,z]$ with $H(0)= 0$. Suppose that $u,~v,~h$ are linearly independent over $K$. If $JH$ is nilpotent, whether there exist $T\in \operatorname{GL}_3(K)$ such that $THT^{-1}$ has the form of Theorem \ref{thm2.3}?
\end{prob}

\end{document}